\documentclass[12pt]{amsart}
\usepackage{amsfonts}
\usepackage{amsmath}
\usepackage{graphicx}
\usepackage{amssymb}
\newtheorem{theorem}{Theorem}[section]
\newtheorem{corollary}[theorem]{Corollary}

\newtheorem{proposition}[theorem]{Proposition}

\theoremstyle{definition}

\theoremstyle{remark}
\newtheorem{remark}{\bf Remark}[section]

\numberwithin{equation}{section}

\def\p{\partial}

\def\wg{\widetilde g}

\begin{document}

\title[Liouville type theorems for the $p$-harmonic functions]{
Liouville type theorems for the $p$-harmonic functions on certain manifolds}

\author{Jingyi Chen and Yue Wang}
\address{ Department of Mathematics\\ The University of British Columbia, Vancouver, BC V6T1Z2, Canada}
\email{jychen@math.ubc.ca}
\address{Department of Mathematics\\China Jiliang University, Hangzhou, Zhejiang, China}
\email{kellywong@cjlu.edu.cn}

\begin{abstract}
We show that the Dirichlet problem at infinity is unsolvable for the $p$-Laplace equation for any nonconstant  continuous boundary data, for certain range of $p>n$, on an $n$-dimensional Cartan-Hadamard  manifold constructed from a  complete noncompact shrinking gradient Ricci soliton. Using the steady gradient Ricci soliton, we find an incomplete Riemannian metric on  ${\mathbb R}^2$ with positive Gauss curvature such that every positive $p$-harmonic function must be constant
for $p\geq 4$.
\end{abstract}

\date{October 21, 2014}
\thanks{The first author is partially supported by NSERC (RGPIN 203199-1). The second author is supported in part by the National Natural Science Foundation of China (No.10901147), and Zhejiang Provincial Natural Science Foundation of China (No.LY12A01028) and she is grateful for PIMS/UBC for hosting her visit where part of the work was carried out.}

\maketitle

\section{Introduction}

In this article, we study two questions about  the $p$-Laplace equation on Riemannian manifolds. The first one is  the solvability of the Dirichlet problem at infinity on a negatively curved complete noncompact manifold, and the second one is the Liouville property for positive solutions on ${\mathbb R}^2$ equipped with an incomplete metric with positive Gauss curvature. In both cases, the $n$-dimensional manifold $M$ under consideration is equipped with a Riemannian metric $e^{\frac{2f}{p-n}}g$ where $(M,g,f)$ is a complete gradient Ricci soliton which is shrinking for the first case and steady for the second case.

On a Riemannian manifold, for a constant $p>1$, a function $v$ in $W^{1,p}_{loc}\cap L^\infty_{loc}$  is $p$-harmonic if it is a weak solution to the $p$-Laplacian  equation
\begin{equation}\label{p-harmonic}
\text{div}\,(|\nabla v|^{p-2}\nabla v)=0.
\end{equation}
It is known that $p$-harmonic functions are in $C^{1,\alpha}$ (\cite{Tol} and the reference therein).

The behaviour of harmonic, more generally $p$-harmonic, functions depends on the sign of the curvature of the manifold in an essential way. We will discuss negatively curved and non-negatively curved manifolds separately.

A Cartan-Hadamard manifold is a complete simply connected Riemannian manifold with nonpositive sectional curvature everywhere. It is well-known that a Cartan-Hadamard manifold $M$ can be compactified by attaching a sphere $M(\infty)$ at the infinity. In the cone topology, the compactification is homeomorphic to a closed Euclidean $n$-ball \cite{E-O}. The Dirichlet problem at infinity for $p$-harmonic functions is to solve  the $p$-Laplace equation:
\begin{equation}\label{p-harmonic}
\text{div}\,(|\nabla v|^{p-2}\nabla v)=0
\end{equation}
on $M$ such that $v$ agrees with a given continuous function $\varphi$ on $M(\infty)$. For $p=2$, the Dirichlet problem at infinity for harmonic functions is solvable if there are suitable lower and upper bounds for the sectional curvature (Andersen \cite{Andersen}, Andersen-Schoen \cite{And-Sch}, Choi \cite{Choi},  Hsu \cite{Hsu}, Sullivan \cite{Sul}).
  Ancona \cite{An} constructed an example showing the Dirichlet problem is unsolvable if only a negative constant upper bound is imposed. For $p\in(1,\infty)$,  the Dirichlet problem at infinity is solvable under similar curvature assumptions like those in the case $p=2$ , in particular, it is solvable if the sectional curvature is bounded by
\begin{equation}\label{curvature}
  -r^{2\alpha-4-\epsilon}\leq K\leq -\frac{\alpha(\alpha-1)}{r^2}
\end{equation}
  near $M(\infty)$ where $\epsilon>0,\alpha>1$ where $r$ is the distance to a fixed point, for $p\in(1,1+(n-1)\alpha)$ (Holopainen \cite{Holo}, Holopainen-V\"ah\"akangas \cite{HV}, Pansu \cite{Pansu}).

Our first result is to show the unsolvability of the Dirichlet problem at infinity on certain Cartan-Hadamard manifolds constructed from shrinking gradient Ricci solitons, for certain range of $p>n$, which include the shrinking Gaussian soliton $({\mathbb R}^n, dx^2, {\frac{|x|^2}{4}})$ for every $p>n$. It is interesting to observe that the sectional curvature of the complete negatively curved metric $e^{\frac{|x|^2}{2(p-n)}}dx^2$ is not bounded above by $-\frac{\alpha(\alpha-1)}{r^2}$, for any constant $\alpha>1$,  at certain sections for sufficiently large $r$ (see Remark \ref{sharp}).
This indicates the upper bound in \eqref{curvature} is sharp in some sense for the solvability of the Dirichlet problem at infinity.

\begin{theorem}\label{Dirichlet_Infinity}
Suppose $(M,g,f)$ is a simply connected $n$-dimensional complete noncompact shrinking gradient Ricci soliton whose the sectional curvatures is bounded above by a constant $K_0$ with $0<K_0<\frac{1}{2(n-1)}$. Then the Dirichlet problem at infinity for the $p$-Laplace equation on $(M,e^{\frac{2f}{p-n}}g)$ is unsolvable for any nonconstant  continuous boundary value $\varphi$ and $n<p<\frac{1}{K_0}+2-n$.
\end{theorem}

The proof replies on a Liouville type property (Proposition \ref{p>1}) for positive solutions to the $p$-Laplace equation on $(M,e^{-\frac{2f}{n-p}}g)$ for every $p>1$, where Cao-Zhou's estimates on $f$ and on the volume growth for gradient shrinking Ricci solitons \cite{CZ} are crucial as they imply that $e^{-f}$ is integrable on $(M,g)$. The advantage for considering the range $p>n$ is that, under the conformal change of metric,  it yields a complete metric $\tilde g$ and it guarantees the  negativity of the curvature of $\tilde g$ under the curvature assumption $K\leq K_0$, while one does not have such flexibility for $p=2$.

\vspace{.3cm}

However, the integration argument in the proof of Proposition \ref{p>1} is no longer valid for steady gradient Ricci solitons due to different behaviour of $f$ (typically $f$ tends to $-\infty$ along a sequence of points $x_k$ that go to infinity \cite{MS}, \cite{Wu}). Alternatively, a powerful way to prove Liouville type theorems for positive harmonic functions on complete manifolds with non-negative Ricci curvature is via Yau's gradient estimate \cite{Yau}. The $p$-harmonic version of Yau's estimate is established by Wang-Zhang \cite{WZ} (see \cite{SW} for a sharp form of the estimate). For a positive $p$-harmonic function $u$ in the conformally changed metric $\tilde g=e^{-\frac{2f}{n-p}}g$, we will first derive a maximum principle for $|\nabla \log u|$ for steady (or shrinking) gradient Ricci solitons, via a Bochner type formula. However, the required assumption on Ricci curvature for the gradient estimates cannot hold globally for steady gradient Ricci solitons if $\dim M >2$ because it would imply the scalar curvature of $g$ possesses a positive constant lower bound but this is impossible as shown in \cite{MS} and \cite{Wu}. In dimension 2, we can combine the maximum principle (Proposition \ref{max}) and the gradient estimate to prove a Liouville type result on the 2-plane with a positively curved {\it incomplete} metric.

\begin{theorem}\label{cigar}
Let $({\mathbb R}^2, g,f)$ be Hamilton's cigar soliton. Then there does not exist any nonconstant  positive $p$-harmonic function on $({\mathbb R}^2, \tilde g)$ for $p\geq 4$.
\end{theorem}

Harmonic functions on the complete gradient Ricci solitons have been studied by Munteanu-Sesum \cite{MS} and  Munteanu-Wang \cite{MW} with applications to the geometry and topology of the solitons; Moser \cite{Moser}
observed an interesting connection between the inverse mean curvature flow formulated as level sets in ${\mathbb R}^n$ and 1-harmonic functions; Kotschwar-Ni \cite{KN} generalize this to Riemannian ambient manifolds.

\vspace{.2cm}

The first author would like to thank Jiaping Wang for valuable discussions and his interest in this work.

\section{The Dirichlet problem at infinity}

In this section, the triple $(M,g,f)$ is assumed to be a complete noncompact shrinking gradient Ricci soliton. We first establish the following Liouville property for positive $p$-harmonic functions for $p>1$ without additional curvature assumption.

An $n$-dimensional Riemannian manifold $(M,g)$ is gradient Ricci soliton if
\begin{equation}\label{Ricci soliton}
Ric +\nabla\nabla f +\varepsilon g =0
\end{equation}
for some smooth function $f$ and $\varepsilon = -\frac{1}{2}, 0, \frac{1}{2}$. Corresponding to the three values of $\varepsilon$, the gradient Ricci soliton $(M,g,f)$ is shrinking, steady, or expanding ({\cite{CLN}, \cite{Ha}).

\begin{proposition}\label{p>1}
Let $(M,g,f)$ be a complete noncompact gradient shrinking Ricci soliton. Then there is no nonconstant  positive $p$-harmonic function on $(M, e^{-\frac{2f}{n-p}}g)$ for $p>1$.
\end{proposition}
\proof Since $u$ is a $p$-harmonic function on $(M, \widetilde g)$ where $\tilde g = e^{-\frac{2f}{n-p}}$,  it holds
\begin{equation}\label{p-tilde}
\mathrm{div}_{\widetilde g}\, \left( |\widetilde\nabla w|^{p-2}\widetilde \nabla w\right) = |\widetilde\nabla w|^p
\end{equation}
where $w=-(p-1)\log u$. For any smooth cut-off function $\phi\in C^\infty_0(M)$, in the complete metric $g$, we require
$$
 \begin{cases}
\phi = 1,   \hspace{2cm} \mbox{on $B_{x_0}(\rho, g)$} \\
\phi = 0, \hspace{2cm} \mbox{on $M\setminus B_{x_0}(2\rho,g)$}\\
0 \leq \phi \leq 1, \hspace{1.22cm}\mbox{on $M$} \\
 |\nabla \phi |^2 \leq  C/ \rho^2, \hspace{0.55cm} \mbox{on $M$.}
\end{cases}
$$
Here $B_{x_0}(r,g)$ stands for the geodesic ball centred at $x_0$ with radius $r$ in the metric $g$ in $M$. Multiplying $\phi^2$ to \eqref{p-tilde} then integrating and applying the Stokes' theorem, we have
\begin{eqnarray*}
\int_M |\widetilde \nabla w|^p_{\wg}\, \phi^2 \,d\mu_{\widetilde g}&=&-2\int_M  \phi|\widetilde\nabla w|^{p-2}_{\wg}\,\widetilde \nabla w\widetilde\nabla\phi \,d\mu_{\widetilde g} \\
&\leq&2\left(\int_M \phi^2|\widetilde\nabla w|^p_{\wg}\,d\mu_{\widetilde g} \right)^{\frac{p-1}{p}}\left(\int_M \phi^{2} |\widetilde\nabla\phi|^p_{\wg}\,d\mu_{\widetilde g} \right)^{\frac{1}{p}}
\end{eqnarray*}
by the Cauchy-Schwarz inequality ($p>1$). Therefore,  we have
$$
\int_M \phi^2|\widetilde \nabla w|^p_{\wg} \, d\mu_{\widetilde g} \leq 2^p \int_M \phi^2 |\widetilde\nabla \phi|^p_{\wg}\, d\mu_{\widetilde g}.
$$
Converting back to the metric $g$, we are led to
\begin{equation}\label{-f}
\int_M \phi^2 |\nabla w|^p e^{-f}\, d\mu_g \leq 2^p \int_M \phi^2 |\nabla \phi|^p e^{-f}\,d\mu_g.
\end{equation}

By Theorem 1.1 in \cite{CZ}, the potential function $f$ for a shrinking gradient Ricci soliton satisfies the pointwise estimate
\begin{equation}\label{f-estimate}
\frac{1}{4}(r(x)-c)^2\leq f(x) \leq \frac{1}{4}(r(x)+c)^2
\end{equation}
for $x \in M\setminus B_{x_0}(1,g)$, where $r(x)$ is the distance from $x$ to a fixed point $x_0$ in $M$ and $c$ is a positive constant.

Therefore, by \eqref{-f} and \eqref{f-estimate}
\begin{eqnarray*}
\int_{B(x_0,\rho)}|\nabla w|^p e^{-\frac{1}{4}(r+c)^2}\, d\mu_g
&\leq& \int_M  \phi^2 |\nabla w|^p e^{-f}\, d\mu_g\\
&\leq& \frac{2^p C e^{-\frac{1}{4}(\rho-c)^2}}{\rho^p}\int _{B_{x_0}(2\rho,g)\setminus B_{x_0}(\rho,g)} d\mu_g \\
&\leq& \frac{2^p C e^{-\frac{1}{4}(\rho-c)^2}}{\rho^p} \rho^n
\end{eqnarray*}
where the last inequality follows from the volume growth estimate (Theorem 1.2 in \cite{CZ}) on shrinking gradient Ricci solitons:
$$
\mathrm{Vol}\,(B_{x_0}(\rho,g))\leq C \rho^n
$$
for sufficiently large $\rho$ and $C$ stands for some uniform constant.
Now letting $\rho\to\infty$, we conclude $|\nabla w|\equiv 0$ on $M$, so $u$ is a constant. \endproof

Next, we show that $(M,\tilde g)$ can be turned into a negatively curved manifold under suitable assumptions on $p$ and the sectional curvature $K$ of $(M,g)$:

\begin{proposition}\label{K}
Let $(M,g,f)$ be a simply connected $n$-dimensional complete noncompact shrinking gradient Ricci soliton whose the sectional curvature is bounded above by a constant $K_0$ with $0<K_0<\frac{1}{2(n-1)}$. Then $(M, e^{-\frac{2f}{n-p}}g)$ is a Cartan-Hadamard manifold for
$n<p\leq \frac{1}{K_0}+2-n$.
\end{proposition}
\begin{proof}
When $p>n$, the metric $\tilde g= e^{-\frac{2f}{n-p}}g$ is complete since
$$
-\frac{2f(x)}{n-p}=\frac{2f(x)}{p-n}\geq \frac{(r-c)^2}{2(p-n)}
$$
by \cite{CZ} and that $g$ is complete.

We use the conventions in \cite{CLN} for curvatures. The Riemann curvature tensor is written as
\begin{eqnarray*}
R\left(\frac{\p}{\p x^i},\frac{\p}{\p x^j}\right)\frac{\p}{\p x^k} &=& R^l_{ijk}\frac{\p}{\p x^l}\\
R_{ijkl} &=& \left<  R\left(\frac{\p}{\p x^i},\frac{\p}{\p x^j}\right)\frac{\p}{\p x^k} , \frac{\p}{\p x^l} \right>
\end{eqnarray*}
and if $\frac{\p}{\p x^1},\cdots, \frac{\p}{\p x^n}$ is orthonormal at $x_0\in M$,  the sectional curvature of the plane $P_{ij}$ spanned by $\frac{\p}{\p x^i},\frac{\p}{\p x^j}$ at $x_0$ is
$$
K(P_{ij}) = R_{ijji}
$$
and the Ricci curvature at $x_0$ is
$$
R_{jk}=\sum^n_{i=1}R^i_{ijk}.
$$
Under the conformal change of metric $\tilde{g}= e^{\frac{2f}{p-n}}g$, the sectional curvature at $x_0$ changes as (cf. p.27 in \cite{CLN}):
\begin{eqnarray}\label{sectional}
\widetilde{K}(P_{ij})&=&\frac{\tilde{g}(\widetilde{R}^s_{ijj}\frac{\p}{\p x^s},\frac{\p}{\p x^i})}{\tilde{g}_{ii}\tilde{g}_{jj}-\tilde{g}_{ij}^2}  \\
&=&e^{\frac{4f}{n-p}}\widetilde{R}_{ijji} \nonumber \\
&=&e^{\frac{4f}{n-p}} \cdot e^{\frac{2f}{p-n}}\left(R_{ijji}-\frac{f_{ii}+f_{jj}}{p-n}-\frac{|\nabla f|^2-f^2_i-f^2_j}{(p-n)^2}   \right) \nonumber \\
&=&e^{\frac{2f}{n-p}}\left(K(P_{ij})-\frac{f_{ii}+f_{jj}}{p-n}-\frac{|\nabla f|^2-f^2_i-f^2_j}{(p-n)^2}  \nonumber  \right).
\end{eqnarray}
On the gradient shrinking Ricci soliton, we therefore have
$$
\tilde{K}(P_{ij})\leq e^{\frac{2f}{n-p}} \left( K(P_{ij}) +\frac{R_{ii}+R_{jj} -1}{p-n}  \right)
$$
by using the defining equation for shrinking gradient Ricci solitons and dropping the last term above that is nonpositive for $i\not=j$.

From the assumption on  $K_0$ and $p>n$, it follows
\begin{eqnarray*}
K(P_{ij}) +\frac{R_{ii}+R_{jj} -1}{p-n} &=& K(P_{ij})+ \frac{\sum_{s\not= i}K(P_{is})+\sum_{s\not= j}K(P_{sj})-1}{p-n}\\
&\leq& \left( 1+\frac{2(n-1)}{p-n}\right) K_0 - \frac{1}{p-n}\\
&\leq& \frac{1}{p-n}\left( \left(p+n-2\right)K_0-1\right).
\end{eqnarray*}
Therefore the sectional curvature $\widetilde K$ of $(M, e^{\frac{2f}{p-n}}g)$ is nonpositive since $p+n-2\leq \frac{1}{K_0}$. \end{proof}

\noindent{\textsl{Proof of Theorem \ref{Dirichlet_Infinity}}. Suppose there is a solution $u$ to the Dirichlet problem at infinity and $u=\varphi$ on $M(\infty)$ for some nonconstant function $\varphi\in C^0(M(\infty))$. Then $u$ is continuous on $M\cup M(\infty)$ hence it is bounded. Then $u-\inf_M u +1$ is a positive solution to the $p$-Laplace equation on $(M,\tilde g)$, therefore it must be constant from Proposition \ref{p>1}.
In turn, $u$ is constant on $M$ and $\varphi$ must be constant on $M(\infty)$. The contradiction concludes the proof.

\

When ${\mathbb R}^n$ is viewed as a shrinking gradient Ricci soliton with $f(x)= \frac{|x|^2}{4}$, we can take $K_0=0$ and we have
\begin{corollary}
The Dirichlet problem at infinity for the $p$-Laplace equation is unsolvable on $({\mathbb R}^n, e^{\frac{|x|^2}{2(p-n)}} dx^2)$ for every $p>n$.
\end{corollary}

\begin{remark}\label{sharp}
The sectional curvature of $\tilde g = e^{2\cdot \frac{|x|^2}{4(p-n)}}dx^2$ can be computed from \eqref{sectional}
\begin{eqnarray*}
\widetilde K(P_{ij})(x)&=& - \, e^{-\frac{|x|^2}{2(p-n)}}\left(\frac{1}{p-n}+\frac{|x|^2-(x^i)^2-(x^j)^2}{4(p-n)^2}\right)
\end{eqnarray*}
where $P_{ij}(x)$ is the plane spanned by $\{\frac{\p}{\p x^i},\frac{\p}{\p x^j}\}$ at $x\in{\mathbb R}^n$. The Riemannian distance from $x$ to the origin is
$$
r(x) = \int^{|x|}_0 e^{\frac{s^2}{4(p-n)}}\, ds.
$$
If we take $x=(0,...,0,x^i,0,...,0)$, then $|x|^2-(x^i)^2-(x^j)^2=0$ and
$$
\lim_{|x|\to\infty}-\widetilde K(P_{ij}(x))\, r^2(x)= \lim_{|x|\to\infty}\frac{ \left(\int^{|x|}_0 e^{\frac{s^2}{4(p-n)}}\, ds\right)^2}{(p-n)e^{\frac{|x|^2}{2(p-n)}}} =\frac{1}{p-n}\left(\lim_{|x|\to\infty}\frac{2(p-n)}{|x|}\right)^2=0
$$
by l'H\^opital's rule. This in particular shows that there does not exist constant $\alpha>1$ which makes true
$$
K(x)\leq -\,\frac{\alpha(\alpha-1)}{r^2(x)}
$$
for all sections at $x$ for large $r(x)$.
\end{remark}

\section{A Liouville theorem on ${\mathbb R}^2$ with an incomplete metric with positive curvature}

In this section, we consider the $p$-Laplace equation weighted by a smooth function $f$ on a manifold $(M,g)$, which is equivalent to
the $p$-Laplace equation on $(M,e^{-\frac{2f}{n-p}}g)$ and derive a Bochner formula for its solutions. Specialized to the shrinking or steady gradient Ricci solitons, the Bochner formula yields a maximum principle, and this is applied to Hamilton's cigar soliton.

\subsection{ A Bochner type formula for the weighted $p$-Laplace equation }

Let  $g$ be a Riemannian metric on an $n$-dimensional manifold $M$, $f$ is a smooth real
valued function on $M$. Consider the following equation
\begin{equation}\label{f-p-harmonic}
\text{div}\,(|\nabla u|^{p-2}\nabla u)-|\nabla u|^{p-2}\langle \nabla
f,\nabla u\rangle =0
\end{equation}
on $M$. Equation \eqref{f-p-harmonic} has a variational structure, in fact, it is the Euler-Langrange equation of the
following weighted $p$-energy functional
$$
E_{p,f}(u)=\int_{M}|\nabla u|^{p} e^{-f} d\mu_g.
$$
We call (\ref{f-p-harmonic}) the $f$-weighted $p$-Laplacian equation on $(M,g)$.

\begin{proposition}\label{equivalent}
 Under a conformal change
$\widetilde g=e^{-\frac{2f}{n-p}}g$, $u$ is a solution to $(\ref{f-p-harmonic})$ on $(M,g)$ if and only if
$u$ is a solution to the $p$-Laplace equation $(\ref{p-harmonic})$ on $(M,\widetilde g)$.
\end{proposition}

\begin{proof} We write $\nabla$ for $\nabla_g$ and $\widetilde\nabla$ for $\nabla_{\widetilde g}$. For any $\varphi\in C^\infty_0(M)$,
\begin{eqnarray*}
\int_M \langle \widetilde{\nabla} \varphi, |\widetilde{\nabla} u|^{p-2}_{\widetilde g}\widetilde \nabla u \rangle_{\widetilde g}\,d\mu_{\widetilde g}
&=& \int_M |\widetilde\nabla u|^{p-2}_{\widetilde g}\langle \widetilde\nabla \varphi, \widetilde\nabla u \rangle_{\widetilde g}\,d\mu_{\widetilde g} \\
&=& \int_M \left(e^{\frac{p-2}{n-p}f}|\nabla u|^{p-2}_{ g}\right) e^{\frac{2f}{n-p}}\langle {\nabla}\varphi,{\nabla} u\rangle_{g} \,e^{-\frac{nf}{n-p}}\,d\mu_{g}\\
&=&\int_M  \langle\nabla\varphi,|\nabla u|^{p-2}_g\nabla u\rangle_g \, e^{-f}\,d\mu_g.
%&=& -\int_M \varphi \left( \text{div}_g\,(|\nabla u|_g^{p-2}\nabla u)-|\nabla u|^{p-2}_g\langle \nabla f,\nabla u\rangle_g \right)\, e^{-f}\,d\mu_g
\end{eqnarray*}
This shows that any weak solution to $(\ref{f-p-harmonic})$ on $(M,g)$ is also a weak solution to $(\ref{p-harmonic})$ on $(M,\widetilde g)$ and vice versa.  \end{proof}

Suppose $u(x,t)$ is a positive solution of (\ref{f-p-harmonic}).  Define
\begin{eqnarray*}
w&=&-(p-1)\log u \\
h&=&|\nabla w|^{2}.
\end{eqnarray*}
We consider the symmetric $n\times n$ matrix
$$
A=\text{id}+(p-2)\frac{\nabla w \otimes   \nabla w }{h}.
$$
Note that $A$ is well defined whenever $h>0$ and  is positive definite for $p>1$. Arising from the linearized operator of the nonlinear $p$-harmonic equations, this matrix was first introduced in \cite{Moser} and was used in \cite{KN} and \cite{WZ} to study positive $p$-harmonic functions.

For the $f$-weighted $p$-Laplace equation \eqref{f-p-harmonic}, the linearized operator is
$$
{\mathcal L}(\psi)= \mathrm{div}\,\left(h^{\frac{p}{2}-1}A(\nabla \psi)\right)-h^{\frac{p}{2}-1}\langle\nabla f, A(\nabla \psi)\rangle-ph^{\frac{p}{2}-1}\langle\nabla w,\nabla \psi \rangle
$$
for smooth functions $\psi$ on $M$ and the following Bochner type formula holds.

\begin{proposition}\label{Bochner}
 Let $u$ be a positive smooth solution to \eqref{f-p-harmonic} in an open subset $U$ in $M$ and assume $h>0$ on $U$. Then
\begin{eqnarray}\label{Bochner formula}
&&\mathrm{div}\,\left(h^{\frac{p}{2}-1}A(\nabla h)\right)-h^{\frac{p}{2}-1}\langle\nabla f, A(\nabla h)\rangle-ph^{\frac{p}{2}-1}
\langle\nabla w,\nabla h\rangle\\
&=&\left(\frac{p}{2}-1\right)|\nabla h|^{2}
h^{\frac{p}{2}-2}+2h^{\frac{p}{2}-1} \left( | \nabla\nabla w |^{2} +
Ric(\nabla w, \nabla w)+ \nabla\nabla f (\nabla w,\nabla w) \right). \nonumber
\end{eqnarray}
\end{proposition}
\proof Using \eqref{f-p-harmonic}, we first observe
\begin{eqnarray}\label{eqn:w}
\ \ \ \ \ \ \ \ \text{div}\,(|\nabla w|^{p-2}\nabla w)-|\nabla w|^{p}
&=&-(p-1)^{p-1}\text{div}\, \left(\frac{|\nabla u|^{p-2}\nabla u}{u^{p-1}} \right) -(p-1)^p\frac{|\nabla u|^p}{u^p} \\
&=&-(p-1)^{p-1}\frac{|\nabla u|^{p-2}\langle\nabla f,\nabla u\rangle}{u^{p-1}} \nonumber\\
&=&|\nabla w|^{p-2}\langle \nabla f,\nabla w\rangle. \nonumber
\end{eqnarray}
Then we calculate directly
\begin{eqnarray*}
\text{div}\,\left(h^{\frac{p}{2}-1}A(\nabla h)\right)&=&\left(\frac{p}{2}-1\right)h^{\frac{p}{2}-2}|\nabla h|^{2}+h^{\frac{p}{2}-1}\Delta h\\
& &+\left(\frac{p}{2}-2\right)\left(p-2\right)h^{\frac{p}{2}-3}\langle\nabla w,\nabla h\rangle^{2}+\left(p-2\right)h^{\frac{p}{2}-2}\langle\nabla w,\nabla h\rangle\Delta w\\
& &+\left(p-2\right)h^{\frac{p}{2}-2}\langle\nabla\langle\nabla
w,\nabla h\rangle,\nabla w\rangle.
\end{eqnarray*}
Using the standard Bochner type formula for $h=|\nabla w|^2$:
$$
\Delta h=2|\nabla\nabla w|^{2}+2Ric\left(\nabla w, \nabla
w\right)+2\langle \nabla \Delta w,\nabla w \rangle
$$
we have
\begin{eqnarray}\label{1}
&& \text{div}\,\left(h^{\frac{p}{2}-1}A(\nabla
h)\right)=\left(\frac{p}{2}-1\right)h^{\frac{p}{2}-2}|\nabla
h|^{2} \\
&&+2h^{\frac{p}{2}-1}\left(|\nabla\nabla w|^{2}+Ric\left(\nabla w,
\nabla w\right)
+\langle\nabla\Delta w,\nabla w\rangle\right)  \nonumber\\
&& +\left(\frac{p}{2}-2\right)(p-2)h^{\frac{p}{2}-3}\langle\nabla w,\nabla h\rangle^{2}+(p-2)h^{\frac{p}{2}-2}
\langle\nabla w,\nabla h\rangle\Delta w \nonumber \\
& &+(p-2)h^{\frac{p}{2}-2}\langle\nabla\langle\nabla w,\nabla h\rangle,\nabla w\rangle. \nonumber
\end{eqnarray}
Rewrite \eqref{eqn:w} by using $h=|\nabla w|^2$  as
\begin{equation}\label{eqn:h}
h^{\frac{p}{2}-1}\Delta w + \left(\frac{p}{2}-1\right) h^{\frac{p}{2}-2}\langle \nabla h, \nabla w \rangle
-h^{\frac{p}{2}} = h^{\frac{p}{2}-1}\langle\nabla f,\nabla w\rangle.
\end{equation}
Taking the gradient of both sides of (\ref{eqn:h}) and then taking the product with $\nabla w$, we are led to

\begin{eqnarray}\label{2}
& &\left(\frac{p}{2}-1\right)\left(\frac{p}{2}-2\right)h^{\frac{p}{2}-3}\langle\nabla w,\nabla h\rangle^{2}
+\left(\frac{p}{2}-1\right)h^{\frac{p}{2}-2}\langle\nabla\langle\nabla w,\nabla h\rangle,\nabla w\rangle\\
&&+h^{\frac{p}{2}-1}\langle\nabla\Delta w,\nabla w\rangle
+\left(\frac{p}{2}-1\right)h^{\frac{p}{2}-2}\langle\nabla h,\nabla w\rangle\Delta w
-\frac{p}{2}h^{\frac{p}{2}-1}\langle\nabla h,\nabla w\rangle \nonumber \\
&=&\left(\frac{p}{2}-1\right)h^{\frac{p}{2}-2}\langle\nabla f,\nabla
w\rangle\langle\nabla h, \nabla w\rangle
+h^{\frac{p}{2}-1}\langle\nabla\langle\nabla f,\nabla w\rangle,\nabla w\rangle. \nonumber
\end{eqnarray}
Adding \eqref{1} and twice of \eqref{2} together and then simplifying, we have
\begin{eqnarray}\label{eqn:div}
&&\text{div}\, \left(h^{\frac{p}{2}-1}A(\nabla h)\right) -p h^{\frac{p}{2}-1}\langle\nabla h,\nabla w\rangle \\
&=& \left(\frac{p}{2}-1\right)h^{\frac{p}{2}-2}|\nabla h|^2+2h^{\frac{p}{2}-1}|\nabla\nabla w|^2 +2 h^{\frac{p}{2}-1}Ric\langle\nabla w,\nabla w\rangle \nonumber \\
&&+(p-2)h^{\frac{p}{2}-2}\langle\nabla f,\nabla w\rangle\langle\nabla h,\nabla w\rangle+2h^{\frac{p}{2}-1}\langle\nabla\langle\nabla f,\nabla w\rangle,\nabla w\rangle. \nonumber
\end{eqnarray}
We also have
\begin{eqnarray}\label{3}
\ \ \ \ \ \ 2h^{\frac{p}{2}-1}\langle\nabla\langle\nabla f,\nabla w\rangle,\nabla w\rangle
&=&  2h^{\frac{p}{2}-1}\left[(\nabla w)(\nabla\nabla f)(\nabla w)'   + (\nabla f )(\nabla\nabla w)( \nabla w)'\right]  \\
&=&2h^{\frac{p}{2}-1}(\nabla w)(\nabla\nabla f)(\nabla w)'+h^{\frac{p}{2}-1}\langle\nabla f,\nabla|\nabla w|^2\rangle
\nonumber \\
&=&2h^{\frac{p}{2}-1}(\nabla w)(\nabla\nabla f)(\nabla w)'+h^{\frac{p}{2}-1}\langle\nabla f,\nabla h\rangle
\nonumber
\end{eqnarray}
where the gradient of a function is a row vector and the prime denotes its transpose. Moreover,
\begin{eqnarray}\label{4}
\ \ h^{\frac{p}{2}-1}\langle\nabla f,A(\nabla h)\rangle &=&
h^{\frac{p}{2}-1}\langle\nabla f,\nabla h\rangle +(p-2)h^{\frac{p}{2}-2}\langle\nabla f, (\nabla w\otimes\nabla w)\nabla h\rangle\\
&=&h^{\frac{p}{2}-1}\langle\nabla f,\nabla h\rangle +(p-2)h^{\frac{p}{2}-2}\langle\nabla f,\nabla w\rangle\langle\nabla h,\nabla w\rangle. \nonumber
\end{eqnarray}
Now, $\eqref{eqn:div} - \eqref{4} + \eqref{3}$ yields the desired result.  \endproof

\subsection{A maximum principle}

When the triple $(M,g,f)$ is either shrinking or steady, Proposition \ref{Bochner} can be used to prove the following maximum principle.
\begin{proposition}\label{max}
Let $u$ be a positive smooth solution to \eqref{f-p-harmonic} in a bounded connected open subset $U$ in $M$ with smooth boundary $\partial U$, $p> 1$.  Suppose $(M,g,f)$ is a shrinking or steady gradient Ricci soliton. Then $\frac{|\nabla u|}{u}$ attains its maximum on  $\partial U$.
\end{proposition}
\proof Let $h=(p-1)^2\frac{|\nabla u|^2}{u^2}$. Assume $\max_{\overline{U}}h >\max_{\partial U}h$. Then, there exists $x_0 \in U$ such that $h(x_0)=\max_{\overline{U}}h >0$. Since $u\in C^{1,\alpha}$ and $u>0$, $h$ is continuous.  Let
$$
V=\{ x\in U: h(x) = h(x_0)\}.
$$
 By the continuity of $h$, $V$ is a closed subset of $U$ and $V$ does not intersect $\partial U$. There exists a point $x_1\in V$ such that the geodesic ball $B_{x_1}(r,g)\subset U$ is not contained in $V$ for any $0<r <r_0$ for some $r_0$, i.e. $x_1$ is a boundary point of $V$. By the continuity of $h$ again,
 there is geodesic ball $B_{x_1}(r_1,g)$ in $U$ on which $h$ is positive. Observe
\begin{eqnarray*}
{\mathrm{RHS \ of } \ \eqref{Bochner formula}}&=& \frac{p-2}{2}|\nabla h|^{2}
h^{\frac{p}{2}-2}+2h^{\frac{p}{2}-1}|\nabla\nabla w|^{2}
+2h^{\frac{p}{2}-1}\left(Ric+\nabla\nabla f\right)\left(\nabla w, \nabla w\right)  \\
&\geq&  2h^{\frac{p}{2}-1}\left(Ric+\nabla\nabla f\right)\left(\nabla w, \nabla w\right) \\
&=&\begin{cases}
2 h^{\frac{p}{2}-1}|\nabla w|^{2}\geq 0, &\quad \mbox{if $(M,g,f)$ is a shrinking soliton}\\
 0, &\quad \mbox{ if $(M,g,f)$ is a steady soliton}
\end{cases}
\end{eqnarray*}
where for the first inequality, we argue as
\begin{eqnarray*}
4h |\nabla\nabla w|^2+(p-2)|\nabla h|^2&\geq&4|\nabla w|^2 |\nabla\nabla w|^2 - |\nabla |\nabla w|^2|^2\\
&=&4|\nabla w|^2 \left(|\nabla\nabla w|^2 - |\nabla |\nabla w||^2\right)\\
&\geq& 0
\end{eqnarray*}
by Ito's inequality and $p\geq 1$.
Then it follows that the linear differential operator ${\mathcal L}$  satisfies
$
{\mathcal L}(h)\geq 0
$
on $U$. Next, since $A$ is positive definite and symmetric on $B_{x_1}(r_1,g)$, so is $h^{\frac{p}{2}-1}A$; therefore, ${\mathcal L}$ is uniformly elliptic on $B_{x_1}(r_1,g)$.
By Hopf's strong maximum principle (cf. Theorem 3.5 in \cite{GT}), $h$ must be a constant on $B_{x_1}(r_1,g)$ since it attains its maximum at the interior point $x_1$. However, this contradicts the maximality of $V$ as $B_{x_1}(r_1,g)$ contains points not in $V$.   \endproof

\subsection{Gradient estimates}

Let us first recall a gradient estimate in \cite{WZ}:

\begin{theorem}\label{WZ} (Wang-Zhang) Let $(M^{n} , g)$ be a complete Riemannian manifold with  $Ric\geq -(n-1)\kappa$ for some positive constant $\kappa$. Assume that $v$ is a positive
$p$-harmonic function on the geodesic ball $B_{x_0}(R,g)\subset  M$. Then
$$
\frac{|\nabla v|}{v}\leq C(p,n)\left(\frac{1}{R}+\sqrt{\kappa}\right)
$$
on $B_{x_0}(\frac{R}{2},g)$ for some constant $C(p,n)$.
\end{theorem}

We prove the following gradient estimate for the $f$-weighted $p$-Laplacian equation.

\begin{proposition}\label{gradient estimate}
Let $(M^{n} , g, f)$ be a complete gradient Ricci soliton with
\begin{equation}\label{Ricci condition}
\left(\frac{2-p}{n-p}\right)Ric\geq -(n-1)\kappa
e^{-\frac{2f}{n-p}}g-\frac{2\varepsilon g}{n-p}-\frac{S g}{n-p}
-\left(d f\otimes   d f-|\nabla f|^{2}g\right)\frac{n-2}{(n-p)^{2}}
\end{equation}
where $S$ is the scalar curvature of $(M,g)$.
 Assume that $u$
is a positive solution of equation $(\ref{f-p-harmonic})$. Then there exists a constant $C(p,n)$ such that
$$
\frac{|\nabla u(x)|}{u(x)}\leq
C(p,n)\left(\frac{1}{R}+\sqrt{\kappa}\right) e^{-\frac{f(x)}{n-p}}
$$
for $x\in B_{x_0}(\frac{R}{2}, e^{-\frac{2f}{n-p}}g)$.
\end{proposition}

\begin{proof}  For a smooth function $f$, let  $\nabla f$ be  the
gradient, $\Delta f$ the Laplacian and $\nabla\nabla f$ the Hessian
with respect to $g$. For the conformal change of metrics $\widetilde g=e^{-\frac{2f}{n-p}} g$, the Ricci tensors of $\widetilde g$ and $g$ are related by (see [3], page 59):
\begin{equation}\label{Ricci change}
\widetilde{Ric}=Ric-(n-2)\left(-\frac{\nabla \nabla f}{n-p}-\frac{df\otimes df}{(n-p)^{2}}\right) + \left(-\frac{\Delta f}{n-p}
-\frac{n-2}{(n-p)^{2}}|\nabla f|^{2}\right)g.
\end{equation}

From the gradient Ricci soliton equation (\ref{Ricci soliton}), the scalar curvature
$S$ of $M$ satisfying the following two equations (see [4]):
\begin{equation}\label{2nd-order}
S+\Delta f -n\varepsilon=0,\
\end{equation}
\begin{equation}\label{1st-order}
S+|\nabla f|^{2}+\varepsilon f=0.
\end{equation}

Putting (\ref{Ricci soliton}) and (\ref{2nd-order}) into (\ref{Ricci change}), we have
\begin{eqnarray*}
\widetilde{Ric}&=&Ric+(n-2)\left(\frac{-Ric -\varepsilon g}{n-p}+\frac{df\otimes
df}{(n-p)^{2}}\right)+\left(\frac{S+n\varepsilon}{n-p}-\frac{n-2}{(n-p)^{2}}|\nabla f|^{2}\right)g\\
&=&\frac{2-p}{n-p}Ric+\frac{2\varepsilon g}{n-p}+\frac{S g}{n-p}+\left(df\otimes  df-|\nabla
f|^{2}g\right)\frac{n-2}{(n-p)^{2}}.
\end{eqnarray*}
Therefore, the curvature assumption in Theorem \ref{gradient estimate} implies
$$
\widetilde{Ric} \geq -(n-1)\kappa.
$$
By Proposition \ref{equivalent}, we know that $u$ is also a positive solution to (\ref{p-harmonic}) for the metric $\widetilde g$, hence by Theorem \ref{WZ} we have
$$
\frac{|\nabla u|_{\widetilde g}}{u}\leq
C(p,n)\left(\frac{1}{ R}+\sqrt{\kappa}\right)
$$
on $B_{x_0}(\frac{R}{2}, \widetilde g)$. This is equivalent to
$$
\frac{|\nabla u(x)|}{u(x)}\leq C(p,n)\left(\frac{1}{ R}+\sqrt{\kappa}\right)e^{-\frac{f(x)}{n-p}}
$$
for $x\in B_{x_0}(\frac{R}{2}, \widetilde g)$.
\end{proof}

\subsection{A Liouville type theorem for $p$-Laplace equation in dimension 2} For a steady gradient Ricci soliton, the condition \eqref{Ricci condition} on the Ricci curvature in Proposition \ref{gradient estimate} cannot hold globally when $n\geq 3$ because it would imply, by taking trace, that the scalar curvature is bounded below by a positive constant but this is impossible. However, the condition \eqref{Ricci condition} is satisfied when $n=2$
for $p\geq 4$ or $1<p<2$ because
$$
Ric = \frac{1}{2} \, S g \geq \frac{1}{p-2} \, S  g
$$
since $S\geq 0$ for any steady gradient Ricci soliton \cite{BLChen} and $\kappa =0$.

Note that Hamilton's cigar soliton is the unique 2-dimensional complete noncompact steady gradient Ricci soliton. The cigar soliton is ${\mathbb R}^2$ equipped with the complete metric (cf. \cite{CLN}):
$$
g = \frac{dx^2+dy^2}{1+x^2+y^2}
$$
and the potential function
$$
f(x,y) = - \log (1+x^2+y^2).
$$
The conformally altered metric is
$$
\tilde g = e^{-2\frac{-\log(1+x^2+y^2)}{2-p}}g= (1+x^2+y^2)^{\frac{p}{2-p}}(dx^2+dy^2).
$$
In particular, $\tilde g$ is complete if $1<p<2$ and incomplete if $p>2$. However, to use the gradient estimate in proving Liouville type result, we will need $p\geq 4$. It is straightforward to compute the Gauss curvature of $\tilde g$:
\begin{eqnarray*}
\widetilde K&=& - \frac{1}{2}(1+r^2)^{\frac{p}{p-2}}\left(\p^2_{rr}+\frac{1}{r}\p_r\right) \log (1+r^2)^{-\frac{p}{p-2}} \\
&=& \frac{2p}{p-2}(1+r^2)^{\frac{p}{p-2}-2}\\
&=&\frac{2p}{p-2}(1+r^2)^{-\frac{p-4}{p-2}}
\end{eqnarray*}
which is positive and tends to 0 as $r\to\infty$ if $p>4$. When $p=4$, the incomplete metric $(1+x^2+y^2)^{-2}(dx^2+dy^2)$ has constant curvature $\widetilde K = 4$.

\begin{theorem}\label{cigar}
Let $({\mathbb R}^2, g,f)$ be Hamilton's cigar soliton. Then there does not exist any nonconstant  positive $p$-harmonic function on $({\mathbb R}^2, \tilde g)$ for $p\geq 4$.
\end{theorem}
\begin{proof} Let $u$ be a positive solution to \eqref{f-p-harmonic}. For any point $x_0\in M$,  the maximum principle (Corollary \ref{max}) asserts
$$
\frac{|\nabla u(x_0)|}{u(x_0)} \leq \max_{x\in \p B_{0}(R,g)}\frac{|\nabla u(x)|}{u(x)}=\frac{|\nabla u(x_R)|}{u(x_R)}
$$
for some $x_R\in\p B_{0}(R,g)$ where $x_0\in B_0(R,g)$ and $r(x_0,0)<R$.
From the discussion above, when $n=2$ and $p\geq 4$, the Ricci curvature condition \eqref{Ricci condition} in Proposition \ref{gradient estimate} is satisfied.
The diameter of $({\mathbb R}^2,\tilde g)$ is
$$
2R_0=2\int_0^\infty \frac{dr}{(1+r^2)^{\frac{p}{2(p-2)}}}<\infty.
$$
It is clear that  $r(x_R, 0)\to\infty$ if and only if  $\tilde r(x_R, 0) \to R_0$, where $\tilde r$ denotes the distance function for the metric $\tilde g$.
Let
$$
r_{R} = \int^\infty_{R}\frac{dr}{(1+r^2)^{\frac{p}{2(p-2)}}}.
$$
It follows from Proposition \ref{gradient estimate}, applied on the ball $B_{x_R}(r_{R},\tilde g)$, that
\begin{eqnarray*}
\frac{|\nabla u(x_R)|}{u(x_R)} &\leq& C(n,p) \left(\frac{r_{x_R}}{2} \right)^{-1}
 e^{-\frac{2}{p-2}\log(1+|x_R|^2)}\\
 &=&2C(n,p)\left(\int^\infty_{R}\frac{dr}{(1+r^2)^{\frac{p}{2(p-2)}}}(1+R^2)^{\frac{2}{p-2}}\right)^{-1}\\
 &\leq& 2C(n,p)\left((1+R^2)^{\frac{2}{p-2}}\int^\infty_R \frac{dr}{r^{\frac{p}{p-2}}}\right)^{-1}\\
 &=& 2C(n,p)\left(\frac{p-2}{2}(1+R^2)^{\frac{2}{p-2}}R^{-\frac{2}{p-2}}\right)^{-1}.
\end{eqnarray*}
Since $p>2$, letting $R\to 0$ we conclude $|\nabla u(x_0)|=0$, hence $u$ is constant as $x_0$ is arbitrary.
\end{proof}

%\begin{remark}Theorem \ref{cigar} can be stated for $p$-harmonic functions on the open unit disk ${\mathbb D}$ by changing the variable  $r$ to $t = \frac{r}{1+r}$, and in the polar coordinates $(t,\theta)$ on the disk. On the other hand, there are plenty of positive nonconstant  $p$-harmonic functions given by affine functions for every $p$ on $({\mathbb D},dx^2+dy^2)$. \end{remark}


\begin{thebibliography}{2}


%\bibitem{SY}R.Schoen, S.T.Yau, Lectures on Differential Geometry. Inernational Press, Boston, 1994.

\bibitem{An} A. Ancona, {\it  Convexity at infinity and Brownian motion on manifolds with unbounded negative curvature}, Rev. Mat. Iberoamericana {\bf 10} (1) (1994), 189-220.

\bibitem{Andersen} M. Andersen, {\it The Dirichlet problem at infinity for manifolds of negative curvature}, J. Differential Geom. {\bf 18} (4), (1983), 701-721.

\bibitem{And-Sch} M. Andersen and R. Schoen, {\it Positive harmonic functions on complete manifolds of negative curvature}, Ann. of Math. {\bf 121} (1985), 429-461.

\bibitem{B} A.  Besse, Einstein manifolds. Springer Press, New
York,  1987.

\bibitem{CZ} H.-D. Cao and  D. Zhou, {\it On complete gradient shrinking Ricci
solitons},  J. Differential Geom. {\bf 85} (2) (2010), 175-186.

\bibitem{BLChen} B.-L. Chen, {\it Strong uniqueness of the Ricci flow}, J. Differential Geom. {\bf 82} (2009), no. 2, 362-382.

\bibitem{CLN}{B. Chow, P. Lu and L. Ni, Hamilton's Ricci flow, Graduate Studies in Math. Vol. {\bf 77}, American Mathematical Society Science Press,  2006.}

\bibitem{Choi} H. I. Choi, {\it Asymptotic Dirichlet problem for harmonic functions on Riemannian manifolds}, Trans. Amer. Math. Soc. {\bf 281} (2), (1984), 691-716.



\bibitem{E-O} P. Eberlein and B O'Neill, {\it Visibility manifolds}, Pacific J. Math {\bf 46} (1973), 45-109.

\bibitem{GT} D. Gilbarg and N. Trudinger, Elliptic Partial Differential Equations of Second Order, Springer-Verlag, 1998.

\bibitem{Ha} R. Hamilton, {\it The formation of singularities in the Ricci flow}, Surveys in Differential Geometry, International Press, Cambridge, MA 1995.

\bibitem{Holo} I.  Holopainen, {\it  Asymptotic Dirichlet problem for the p-Laplacian on Cartan-Hadamard manifolds}, Proc. Amer. Math. Soc. {\bf 130} (2002), no. 11, 3393?400.

\bibitem{HV} I. Holopainen and A. V\"ah\"akangas,  {\it Asymptotic Dirichlet problem on negatively curved spaces}, J. Anal. {\bf 15} (2007), 63?10.

\bibitem{Hsu} P. Hsu, {\it Brownian motion and Dirichlet problems at infinity},  Ann. Probab. {\bf 31} (2003), no. 3, 1305-1319.

\bibitem{KN} B. Kotschwar and L. Ni, {\it Local gradient estimates of p-Harmonic
functions,1/H-flow, and entropy formula}, Ann. Sci. \'Ec. Norm. Sup\'er. (4) {\bf 42} (2009), no. 1, 1-36.

\bibitem{Moser} R. Moser, {\it  The inverse mean curvature flow and $p$-harmonic functions}, J. Eur. Math. Soc.  {\bf 9 } (2007), no. 1, 77-83.

\bibitem{MS} O. Munteanu and N. Sesum, {\it On Gradient Ricci solitons},
J. Geom. Anal. {\bf 23} (2013), 539-561.

\bibitem{MW} O. Munteanu and J. Wang, {\it Analysis of weighted Laplacian and application to Ricci solitons}, Comm. Anal. Geom. {\bf 20} (2012), no. 1, 55-94.

\bibitem{Pansu} P. Pansu, {\it Cohomologie $L^p$ des vari\'et\'es \`a courbure n\'egative, cas du degr\'e 1}, Conference on Partial Differential Equations and Geometry (Torino, 1988). Rend. Sem. Mat. Univ. Politec. Torino 1989, Special Issue, 95?20 (1990).

\bibitem{Sul} D. Sullivan, {\it The Dirichlet problem at infinity for a negatively curved manifold}, J. Differential Geometry {\bf 18} (1983), 723-732.

\bibitem{SW} C.-J. Sung and J. Wang, {\it Sharp estimate and spectral rigidity for $p$-Laplacian}, Math, Res. Lett. {\bf 21}, no. 4 (2014), 1-20.

\bibitem{Tol} P. Tolksdorf, {\it Regularity for a more general class of quasilinear elliptic equations}, J. of Differential Equations {\bf 51}, 1984, 126-150.

\bibitem{V} A. V\"ah\"akangas, {\it Dirichlet problem at infinity for ${\mathcal A}$-harmonic functions}, Potential Anal. {\bf 27} (2007), no. 1, 27-44.

\bibitem{Yau} S.-T. Yau, {\it Harmonic functions on complete Riemannian manifolds}, Comm. Pure. Appl. Math. {\bf 28} (1975), 201-228.

\bibitem{WZ} X. Wang and  L. Zhang, {\it Local gradient estimate for
p-harmonic functions on Riemannian manifolds}, Comm. Anal. Geom. {\bf 19} (2011), no. 4, 759-771.

\bibitem{Wu} P. Wu, {\it On the potential function of gradient steady Ricci solitons}, J. Geom. Anal. {\bf 23} (2013), no. 1, 221-228.



\end{thebibliography}
\end{document}